\documentclass[english]{extarticle}
\usepackage[T1]{fontenc}
\usepackage[latin9]{inputenc}
\usepackage[a4paper]{geometry}
\usepackage{color}
\usepackage{babel}
\usepackage{float}
\usepackage{mathrsfs}
\usepackage{amsmath}
\usepackage{amsthm}
\usepackage{amssymb}
\usepackage{graphicx}
\usepackage[numbers]{natbib}
\usepackage[unicode=true,pdfusetitle,
 bookmarks=true,bookmarksnumbered=false,bookmarksopen=false,
 breaklinks=false,pdfborder={0 0 1},backref=slide,colorlinks=true]
 {hyperref}
\hypersetup{
 linkcolor=blue, urlcolor=marineblue, citecolor=blue, pdfstartview={FitH}, hyperfootnotes=false, unicode=true}

\def \ff{\infty}
\def \AA{{\cal A}}
\def \BB{{\cal B}}

\def \DD{{\mathbb D}}
\def \CC{{\mathbb C}}

\def \II{{\cal I}}

\def \OO{{\cal O}}

\def \RR{{\mathbb R}}

\def \vski{\vspace{12pt}}

\newcommand{\seg}[2]{\stackrel{\line(1,0){#1}}{#2}}

\def \D{\Delta}
\def \ep{\epsilon}

\def \si{\sigma}

\def \thh{\theta}

\makeatletter
\numberwithin{equation}{section}
\numberwithin{figure}{section}
\theoremstyle{plain}
\newtheorem{thm}{\protect\theoremname}
  \theoremstyle{definition}
  \newtheorem{defn}[thm]{\protect\definitionname}
  \theoremstyle{plain}
  \newtheorem{prop}[thm]{\protect\propositionname}
  \newtheorem{corollary}[thm]{\protect\corollaryname}
  \theoremstyle{definition}
  \newtheorem{example}{\protect\examplename}

\usepackage[T1]{fontenc}
\usepackage{ae,aecompl}
\usepackage[T1]{fontenc}
\usepackage{amsmath}
\usepackage{graphicx}

\usepackage{tikz}

\makeatother

  \providecommand{\definitionname}{Definition}
  \providecommand{\examplename}{Example}
  \providecommand{\propositionname}{Proposition}
  \providecommand{\corollaryname}{Corollary}
\providecommand{\theoremname}{Theorem}

\begin{document}

\title{\textbf{Maximizing the $p$-th moment of exit time of planar Brownian motion from a given domain}}

\author{Maher Boudabra, Greg Markowsky}
\maketitle
\begin{abstract}
In this paper we address the question of finding the point which maximizes the $p$-th moment of the exit time of planar Brownian motion from a given domain. We present a geometrical method of excluding parts of the domain from consideration which makes use of a coupling argument and the conformal invariance of Brownian motion. In many cases the maximizing point can be localized to a relatively small region. Several illustrative examples are presented.
\end{abstract}

\vski


\vski


\section{Introduction}

Let $Z_{t}:=X_{t}+iY_{t}$ be a planar Brownian motion starting at a point $a$ in a domain $U$. We will let $\tau_U = \tau_U(a)$ be the first time that $Z_t$ exits $U$, and we will use the standard notation $E_a$ to denote expectation conditioned on $Z_0 = a$ a.s. The focus of this paper is the following optimization problem.

\vski

{\it For a given domain in the plane and $0< p < \ff$, find the point $a$ which maximizes the quantity $E_a[(\tau_U)^p]$. }

\vski

We will refer to such a point as a {\it $p$-th center} of $U$; it is not in general unique, as the easy example of an infinite strip shows. For many domains, even simple ones such as an isosceles triangle, it is difficult to find any of the $p$-th centers, however we will show how elementary coupling arguments and the conformal invariance of Brownian motion in many cases allows us to locate a small region in $U$ which must contain all $p$-th centers. In certain cases in which the domain in question has a high degree of symmetry, it will allow us to locate all $p$-th centers.

\vski

Before describing our methods, we present a brief overview of some earlier works related to this problem. The case $p=1$ is commonly referred to as the "torsion problem" due to its connection with mechanics, and is naturally the most tractable. The function $h(a) = E_a[\tau_U]$ satisfies $\Delta h = -2$, and therefore p.d.e. techniques can be employed to great effect. \cite[Ch. 6]{sperb} contains a good account of this problem and methods of attacking it in special cases, such as when the domain in question is convex. Further results along the same lines, focussing in particular on convex domains, can be found in \cite{keady,makar,philip}.

\vski

Other interesting related problems have been tackled by p.d.e. methods. For example, in the famous paper \cite{bannydrum} (see also the related work \cite{bannymax}) eigenvalue techniques are used to demonstrate relationships between $E_a[\tau_U]$ and geometric qualities of the domain, such as the size of the hyperbolic density and the inradius (the radius of the largest disk contained in the domain). The methods developed there have been extended by other authors in a number of different directions. For example, in \cite{mendez} a number of related stochastic domination results were proved concerning convex domains in $\RR^n$ and various types of symmetrizations. These results allow conclusions to be reached concerning the comparison of $p$-th moments of the exit times from these domains. One striking consequence of the eigenvalue methods is the fact that over all domains with a given area the disk maximizes the $p$-th moment of the exit time of Brownian motion for all $p$. The recent work \cite{kimmy} contains a discussion and refinements of this result.

\vski

Our results differ from those described above in the following ways. We have not employed p.d.e. methods at all, choosing instead to work with an elementary coupling method. Perhaps as a consequence of this, convexity plays little role in our discussion, although a weaker concept called $\Delta$-convexity (defined below) will be important. The type of coupling we will use is not entirely new, and has found a number of uses in related topics, for instance in investigations into the "hot spots" conjecture such as \cite{bannyhot, bannyhot2, pascu}. However, we believe that it has not been applied directly in the manner that we do before. Furthermore, we restrict our attention to two dimensions, which allows conformal mappings to take prominence and to extend the standard notion of coupling. We present several methods for localizing the $p$-centers of a domain, and then consider a number of specific domains, showing in each case how our methods can be used to localize the $p$-th centers of the domain. In what follows we assume $p$ is a fixed positive number; however, in order to reduce the qualifications needed to state our results, for any planar domain $U$ for which we are interested in maximizing the $p$-th moment we will assume that $E_a[(\tau_U)^p]<\ff$ for all points $a \in U$; this would follow if $E_a[(\tau_U)^p]<\ff$ for any $a \in U$, as is shown in \cite{burk}.

\section{Partial symmetry and convexity with respect to a line.}

\begin{defn}
Let $U$ be a domain of $\mathbb{C}$. We say that a line $\Delta:ax+by+c=0$
is a \emph{partial symmetry axis for $U$} if one of the two sets
$U^{+}:=U\cap\left\{ ax+by+c>0\right\} $ or $U^{-}:=U\cap\left\{ ax+by+c<0\right\} $
can be folded over $\Delta$ and fits into $U$, more precisely if $\sigma_{\Delta}(U^{+})$
or $\sigma_{\Delta}(U^{-})$ remains inside $U$, where $\sigma_{\Delta}$
denotes the symmetry over $\Delta$. The subset among $U^{\pm}$ that
satisfies this property (that is, the smaller side with respect to the symmetry) is called \emph{symmetric side of $U$ over
$\Delta$}. So, for instance, any line intersecting $\DD = \{|z|<1\}$ is a partial symmetry axis for $\DD$, but the line $y=2x$ is not one for the square $\{|x| < 1, |y|<1\}$ since the reflection over this line of the point $(1,1)$ is the point $(\frac{1}{5}, \frac{7}{5})$, which is not in the closure of the square. Notice that both of $U^{\pm}$ are symmetric sides if
and only if $\Delta$ is a symmetry axis for $U$.

\end{defn}

\begin{thm}
\label{prop:part.sym.axis}Let $S$ be the symmetric side of $U$ over
a partial symmetry axis $\Delta$. Then, for any $a \in S$ we can find Brownian motions $Z_t$ starting at $a$ and $\tilde Z_t$ starting at $\sigma_\D(a)$ defined on the same probability space such that $\tau_U \leq \tilde \tau_U$ a.s. (where $\tilde \tau_U$ is the exit time from $U$ of $\tilde Z$). Furthermore, if $\sigma_{\Delta}(S)$ is strictly contained in $U \setminus (S \cup \Delta)$ then $P(\tau_U < \tilde \tau_U) > 0$. In particular, $E_a[\tau_U^p] \leq E_{\si_\D(a)}[\tilde \tau_U^p]$ (with strict inequality if $\sigma_{\Delta}(S)$ is strictly contained in $U \setminus (S \cup \Delta)$).
\end{thm}

\begin{proof}
This follows from a coupling argument. Let $Z_{t}$ start at $a\in S$, and let $H_{\Delta}$ be its
hitting time of the line $\Delta$. Form the process $\tilde Z_t$ by the rule

$$ \tilde Z_t = \left \{ \begin{array}{ll}
\si_\D(Z_t) & \qquad  \mbox{if } t < H_\D  \\
Z_t & \qquad \mbox{if } t \geq H_\D\;.
\end{array} \right. $$

It follows from the Strong Markov property and the reflection invariance of Brownian motion that $\tilde Z_t$ is a Brownian motion. Clearly $\tau_U = \tilde \tau_U$ on the set $\{\tau_U \geq H_\D\}$, and our conditions on $S$ imply $\tau_U \leq \tilde \tau_U$ on the set $\{\tau_U < H_\D\}$. Furthermore, if $\sigma_{\Delta}(S)$ is strictly contained in $U \setminus (S \cup \Delta)$, then $Z_t$ has some positive probability of leaving $U$ before $\tilde Z_t$ does; this is implied for instance by \cite[Thm. I.6.6 ]{bass}. The result follows.
\end{proof}

This theorem allows us in essence to exclude the symmetric side of any partial symmetry axis for $U$ in our search for $p$-th centers. The only exception to this rule is when $\sigma_{\Delta}(S) = U \setminus (S \cup \Delta)$, i.e. when $\D$ is a symmetry axis of $U$; however in most cases a symmetry axis will contain all $p$-th centers. To see why this is so, we need another definition.

\begin{defn}
Let $\Delta$ be a symmetry axis of $U$. We
say that $U$ is $\Delta$-convex if
\[
\forall\left(z,t\right)\in U\times[0,1]\,\,\,tz+(1-t)\sigma_{\Delta}(z)\in U.
\]
In other words, for every $z\in U$, the segment joining $z$ and
$\sigma_{\Delta}(z)$ remains inside $U$.
\end{defn}

It is clear that any convex $U$ is $\Delta$-convex for
any symmetry axis $\Delta$, and a less trivial example can be given by $\{-f(x) < y < f(x)\}$, where $f$ is a positive continuous function on the real line, which is $\D$-convex with $\D = \RR$. A domain which is not $\D$-convex with respect to a symmetry axis can be given by $W_\ep = \{-\ep<y<\ep, |x|<1\} \cup \{|z-1|< \frac{1}{2}\} \cup \{|z+1|< \frac{1}{2}\}$ with $\ep < \frac{1}{2}$; this has the real and imaginary axes as symmetry axes but is not $\D$-convex with respect to the imaginary axis (though it is with respect to the real line). As will be seen below, this domain also shows why $\D$-convexity is required in the following proposition.

\begin{prop}
\label{prop:full.sym.axis} Suppose $U$ is $\D$-convex with respect to a symmetry axis $\D$. Then all $p$-th centers of $U$ lie on $\D$.
\end{prop}

\begin{proof}
Let $a \in U \setminus \D$, and let $\hat a = \frac{1}{2} a + \frac{1}{2} \si_\D(a)$ be the orthogonal projection of $a$ onto $\D$. Let $L$ be the line parallel to $\D$ which passes through the point $\frac{1}{2} a + \frac{1}{2} \hat a$; speaking informally, this is the line halfway between $a$ and $\D$. $\D$-convexity implies that $L$ is a partial symmetry axis of $U$, and if $S$ is the component of $U \setminus L$ containing $a$ then $\si_L(S)$ is strictly contained in $U \setminus (S \cup L)$. It therefore follows from Theorem \ref{prop:part.sym.axis} that $E_a[\tau_U^p] < E_{\hat a}[\tau_U^p]$. The result follows.
\end{proof}

Note that this proposition completely solves our problem in the case that our domain is $\D$-convex with respect to two or more non-parallel symmetry axes, since all $p$-th centers must lie at their point of intersection, and we have also incidentally proved the purely geometrical fact that all such symmetry axes must coincide at a unique point; more on this in the final section. Thus, for instance, all $p$-centers of any regular polygon, a circle, an ellipse, a rhombus, and any number of other easily constructed examples must lie at their natural centers. To see an example of a domain with intersecting symmetry axes but where the point of intersection is not a $p$-th center, let us return to the domains $W_\ep$ described immediately before this proposition. Proposition \ref{prop:full.sym.axis} implies that all $p$-th centers lie on the real line, but it is easy to see that if we make $\ep$ sufficiently small then 0, the intersection point of the two symmetry axes, will not be a $p$-th center (clearly $\tau_{W_\ep}(1) \geq \tau_{\{|z-1|<1/2\}}(1)$, so that $E_1[\tau_{W_\ep}^p]$ remains always greater than a positive constant, but $\tau_{W_\ep}(0)$ decreases monotonically to 0 a.s. as $\ep \searrow 0$, so that $E_1[\tau_{W_\ep}^p] \searrow 0$).

\vski

Let us now look at an example that shows the use of the results proved to this point, but also their limitations. Suppose $U$ is the isosceles right triangle with vertices at $-1,1,$ and $i$. The imaginary axis is an axis of symmetry, and $U$ is $\D$-convex with respect to this axis, so all $p$-th centers must lie on the imaginary axis. The line $\{y=\frac{1}{2}\}$ is a partial symmetry axis for $U$, with $U \cap \{y > \frac{1}{2}\}$ the symmetric side, so all $p$-centers must lie on $\{x=0, y \leq \frac{1}{2}\}$. Now, common sense tells us that the $p$-th centers can't be too close to the real axis as well, because this is a boundary component, but there is no good partial symmetry axis to apply to conclude that rigorously. The way out of this difficulty is to extend our method of reflection to curves more general than straight lines. For this, we will need to utilize the conformal invariance of Brownian motion, via the following famous theorem of L\'evy (see \cite{bass} or \cite{mortper}).

\begin{thm}
\label{Levy}If $f$ is a holomorphic function then $f(Z_{t})$ is a time
changed Brownian motion. More precisely, $f(Z_{\kappa^{-1}(t)})$
is a Brownian motion where
\[
\kappa(t):=\int_{0}^{t}|f'(Z_{s})|^{2}ds
\]
for $t\geq0$.
\end{thm}

This allows us to extend Theorem \ref{prop:part.sym.axis}, as follows.

\begin{prop}
\label{prop:part.sym.axis.conformal} Suppose $U$ is a domain with an axis of symmetry $\Delta$, and suppose $f$ is a conformal map defined on $U$ with the property that $|f'(z)| \geq |f'(\si_\Delta(z))|$ for all $z \in A$, where $A$ is one component of $U \backslash \Delta$ and $\si_\Delta$ is the symmetry over $\Delta$. Then, for any $a \in A$, we can find Brownian motions $Z_t$ starting at $f(a)$ and $\tilde Z_t$ starting at $f(\si_\Delta(a))$ defined on the same probability space such that $\tau_{f(U)}(f(a)) \geq \tilde \tau_{f(U)}(f(\si_\Delta(a)))$ a.s. (where $\tilde \tau_U$ is the exit time from $U$ of $\tilde Z$). In particular, $E_{f(a)}[\tau_{f(U)}^p] \geq E_{f(\si_\Delta( a))}[\tau_{f(U)}^p]$. If there is any point in $A$ at which $|f'(z)| > |f'(\si_\Delta(z))|$ then $E_{f(a)}[\tau_{f(U)}^p] > E_{f(\si_\Delta(a))}[\tau_{f(U)}^p]$ (for this statement we recall the assumption that $E_{w}[\tau_{f(U)}^p]< \ff$ for any $w \in f(U)$).
\end{prop}

\begin{proof}
Let $Z_{t}$ be a Brownian motion starting at $a$, and let $\tilde Z_{t}$ be defined as in Theorem \ref{prop:part.sym.axis}.
According to Theorem \ref{Levy}, the processes $f(Z_t)$ and $f(\tilde Z_{t})$ are time-changed Brownian motions, and the time changes are given by $\kappa^{-1}(s)$ and $\tilde \kappa^{-1}(s)$, respectively, where

$$\left.\begin{aligned}\kappa(t) & =\int_{0}^{t}\mid f'(Z_{t})\mid^{2}dt\\
\tilde \kappa(t) & =\int_{0}^{t}\mid f'(\tilde Z_{t})\mid^{2}dt
\end{aligned}
\;\right\} \; \;t<\tau_{U}.$$


Now, our assumptions imply $|f'(Z_t)| \geq |f'(\tilde Z_{t})|$ a.s. for all $t < \tau_U$, and thus $\kappa(t) \geq \tilde \kappa(t)$ a.s. for all $t < \tau_U$. It follows from this that $\tau \geq \tilde \tau$ a.s., where $\tau$ and $\tilde \tau$ are the exit times from $f(U)$ of the Brownian motions $f(Z_{\kappa^{-1}(s)})$ and $f(\tilde Z_{\tilde \kappa^{-1}(s)})$ which begin at $f(a)$ and $f(\si_\Delta(a))$, respectively. The result follows.
\end{proof}

We can obtain a corollary which will be useful for the isosceles triangle and in other cases by taking $f(z) = z_0 + R\Big(\frac{z+i}{z-i}\Big)$ for $z_0 \in \CC$ and $R>0$, which takes the real axis to the circle $C = \{|z-z_0|=R\}$, the upper half-plane to the outside of $C$, and the lower half-plane to the inside. We have $|f'(z)| = \frac{2R}{|z-i|^2}$, and it is easy to check that $|f'(z)| > |f'(\bar z)|$ for all $z$ in the upper half-plane. Applying Proposition \ref{prop:part.sym.axis.conformal} and working through the implications yields the following.

\begin{corollary} \label{circreflect}
Let $C = \{|z-z_0|=R\}$ be a circle in $\CC$, with inside $\II$ and outside $\OO$. If $U$ is a domain such that $\si_C(U \cap \II) \subseteq (U \cap \OO)$, then no $p$-th center of $U$ lies in $\II$.
\end{corollary}

Note that here $\si_C$ denotes reflection over the circular arc $C$, that is, $\si_C(z) = z_0 + \frac{R^2}{\seg{18}{z-z_0}}$. We remark further that the singularity that $f$ has at $i$ does not cause a problem in this result, because $P_0(Z_t = i \mbox{ for some } t \geq 0) = 0$, so a.s. a Brownian motion starting at $0$ will not hit the singularity, anyway. The compact set $\{B_t:0\leq t \leq \tau_U\}$ is therefore bounded away from $i$ a.s., and the result goes through.

Let us now apply this corollary to the isosceles right triangle $U$ above. If $C$ is the circle passing through $-1$ and $1$ which intersects the real axis at angles of $\frac{\pi}{8}$, then the reflection of the set $\AA=\II \cap U$ will be the region $\BB$ in the upper half-plane bounded by $C$ and the circle which passes through $-1$ and $1$ and intersects the real axis at angles of $\frac{\pi}{4}$; this can be seen by noting that the transformation $\si_C$ preserves angles and and also preserves the class of circles on the Riemann sphere (which includes lines, interpreted as circles through $\infty$).

\begin{figure}[H]
\centering{}\includegraphics[width=9cm,height=9cm,keepaspectratio]{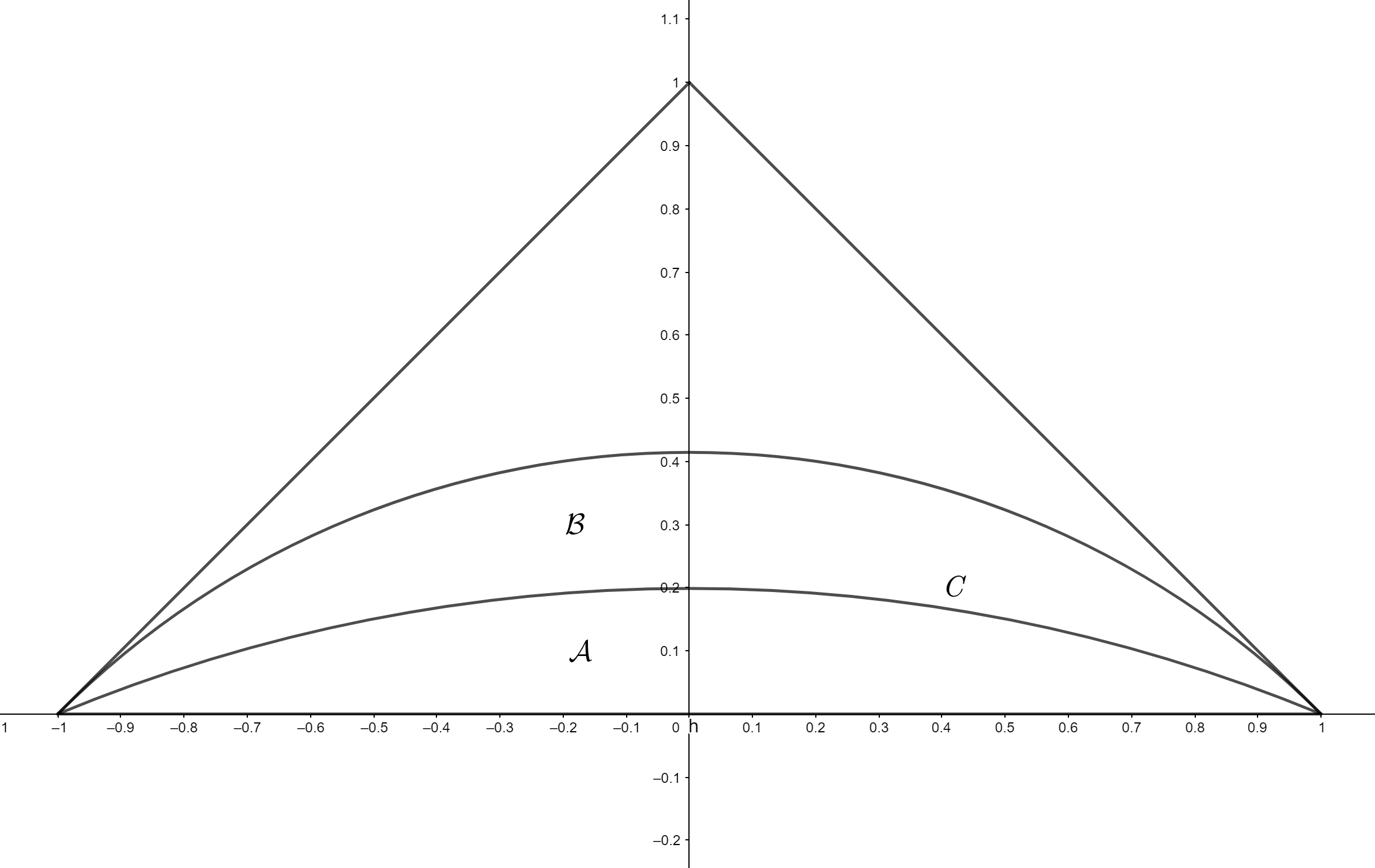}\caption{Reflection over a circle} \label{trireflect}
\end{figure}

As this region lies within $U$, we conclude that no $p$-th centers lie within $\AA$. A bit of Euclidean geometry shows that $C$ intersects the imaginary axis at $\csc(\frac{\pi}{8}) - \cot(\frac{\pi}{8}) \approx .20$, and coupled with our observations above we see that all $p$-th centers must lie on the imaginary axis between the points $.2i$ and $.5i$. In fact, the upper bound of $.5i$ can be improved by using the angle bisector of the angles at $1$ or $-1$, see Example \ref{isotri} in Section \ref{apps}. The reader may also have observed that the reflected circular domain does not do a good job of filling the triangle, and therefore it stands to reason that the lower bound may also be improved; more on this in the final section.

\vski

Finally, the following result can be useful when $U$ is mapped to itself by an {\it antiholomorphic} function $\bar f$ (this means that the conjugate of $\bar f$, $f(z)$, is holomorphic). We will denote the derivative of this function (with respect to $\bar z$ by $\overline{f'(z)}$). An example of this is when $U$ is an annulus, as will be explored in Section \ref{apps}.

\begin{prop}
\label{map U to U} Let $\overline{f}:U\longrightarrow U$ be antiholomorphic and consider
the two following sets
\[
\begin{alignedat}{1}\Omega & :=\{z\in U|\,\,|\bar f'(z)|<1\}\\
\Lambda & :=\{z\in U|\,\,|\bar f'(z)|=1\}.
\end{alignedat}
\]
If $\bar f(U\setminus\Omega)\subset \Omega$ and $\bar f_{|\Lambda}=id_{\Lambda}$,
then all $p$-th centers are contained in $\Omega$.
\end{prop}
\begin{proof}
Let $Z_{t}$ be a Brownian motion starting at $z\in U\setminus\Omega$ and $H_{\Lambda}$ its
hitting time of $\Lambda$, and consider $W_{t}$ the Brownian motion
derived from $f(Z_{t})$, i.e
\[
W_{t}:=f(Z_{\kappa^{-1}(t)})
\]
where
\[
\kappa(t)=\int_{0}^{t}|\bar f'(Z_{s})|^{2}ds.
\]
 Now, we are going to construct two Brownian motions $\widetilde{Z}_{t}$
and $\widetilde{W}_{t}$ starting respectively at $z$ and $w:=f(z)$
such that $\widetilde{W}_{t}$ leaves $U$ before $\widetilde{Z}_{t}$.
as follows :
\begin{enumerate}
\item If $H_{\Lambda}<\tau_{U}^{{\scriptscriptstyle Z}}$ then run an independent
Brownian motion, say $B_{t}$, starting at $Z_{H_{\Lambda}}$and set
\[
\begin{alignedat}{1}\widetilde{Z}_{t} & :=Z_{t}\mathbf{1}_{\{t\leq H_{\Lambda}\}}+B_{t-H_{\Lambda}}\mathbf{1}_{\{H_{\Lambda}<t\}}\\
\widetilde{W}_{t} & :=W_{t}\mathbf{1}_{\{t\leq\kappa(H_{\Lambda})\}}+B_{t-\kappa(H_{\Lambda})}\mathbf{1}_{\{\kappa(H_{\Lambda})<t\}}.
\end{alignedat}
\]
$\widetilde{W}_{t}$ and $\widetilde{Z}_{t}$ are well defined Brownian
motions as
\[
\widetilde{W}_{\kappa(H_{\Lambda})}=W_{\kappa(H_{\Lambda})}=f(Z_{H_{\Lambda}})=Z_{H_{\Lambda}}=\widetilde{Z}_{H_{\Lambda}}.
\]
Now notice that
\[
\begin{alignedat}{1}\tau_{U}^{{\scriptscriptstyle \widetilde{Z}}} & =H_{\Lambda}+\inf\{t,\,\,B_{t}\notin U|B_{0}=Z_{H_{\Lambda}}\}\\
 & \leq\kappa(H_{\Lambda})+\inf\{t,\,\,B_{t}\notin U|B_{0}=Z_{H_{\Lambda}}\}\\
 & =\tau_{U}^{{\scriptscriptstyle \widetilde{W}}}.
\end{alignedat}
\]
\item If $H_{\Lambda}\geq\tau_{U}^{{\scriptscriptstyle Z}}$ then just set
$\widetilde{Z}_{t}=Z_{t}$ and $\widetilde{W}_{t}=W_{t}$. Therefore
\[
\tau_{U}^{{\scriptscriptstyle \widetilde{Z}}}\leq\kappa(\tau_{U}^{{\scriptscriptstyle \widetilde{Z}}})=\tau_{U}^{{\scriptscriptstyle \widetilde{W}}}.
\]
\end{enumerate}
In both cases, we have
\[
\tau_{U}^{{\scriptscriptstyle \widetilde{Z}}}\overset{{\scriptstyle a.s}}{\leq}\tau_{U}^{{\scriptscriptstyle \widetilde{W}}}.
\]
Hence
\[
E_{z}((\tau_{U}^{{\scriptscriptstyle \widetilde{Z}}})^p)\leq E_{w}((\tau_{U}^{{\scriptscriptstyle \widetilde{W}}})^p),
\]
which ends the proof.
\end{proof}

{\bf Remark:} Reflection over the circle, obtained above as a corollary of Proposition \ref{prop:part.sym.axis.conformal}, can just as easily be deduced as a corollary of Proposition \ref{map U to U}.

\section{Applications} \label{apps}

In this section we work through a series of examples that show how our results may be applied.

\begin{example} \label{halfdisk}
Let $U$ be the upper half-disc $\{|z|<1, Im(z)>0\}$. The imaginary axis is an axis of symmetry, and $U$ is $\Delta$-convex with respect to this axis, so all $p$-th centers lie on the imaginary axis. The line $\Delta=:\{Im(z)=\frac{1}{2}\}$ is clearly a partial symmetry axis with symmetric part $U \cap \{Im(z)>\frac{1}{2}\}$, and $U$ is $\Delta$-convex as well. Thus,
all $p$-centers belong
to the set $\{Re(z)=0,0\leq Im(z)\leq\frac{1}{2}\}$. Now, if we let $C$ be the circle $\{|z+i|= \sqrt{2}\}$, then $C$ passes through $1$ and $-1$, making an angle of $\frac{\pi}{4}$ at each point with the real axis. If we let $\AA = U \cap \II$ as before, with $\II$ the inside of $C$, then $\sigma_C(\AA) = \BB$, where $\BB = U \cap \OO$ and $\OO$ is the outside of $C$. By Corollary \ref{circreflect}, no $p$-th center lies in $\AA$. Thus, all $p$-th centers lie on the line segment $\{Re(z)=0,\sqrt{2}-1 \leq Im(z) \leq \frac{1}{2}\}$, which is in bold in the figure below.
\end{example}

\begin{figure}[H]
\centering{}\includegraphics[width=11cm,height=11cm,keepaspectratio]{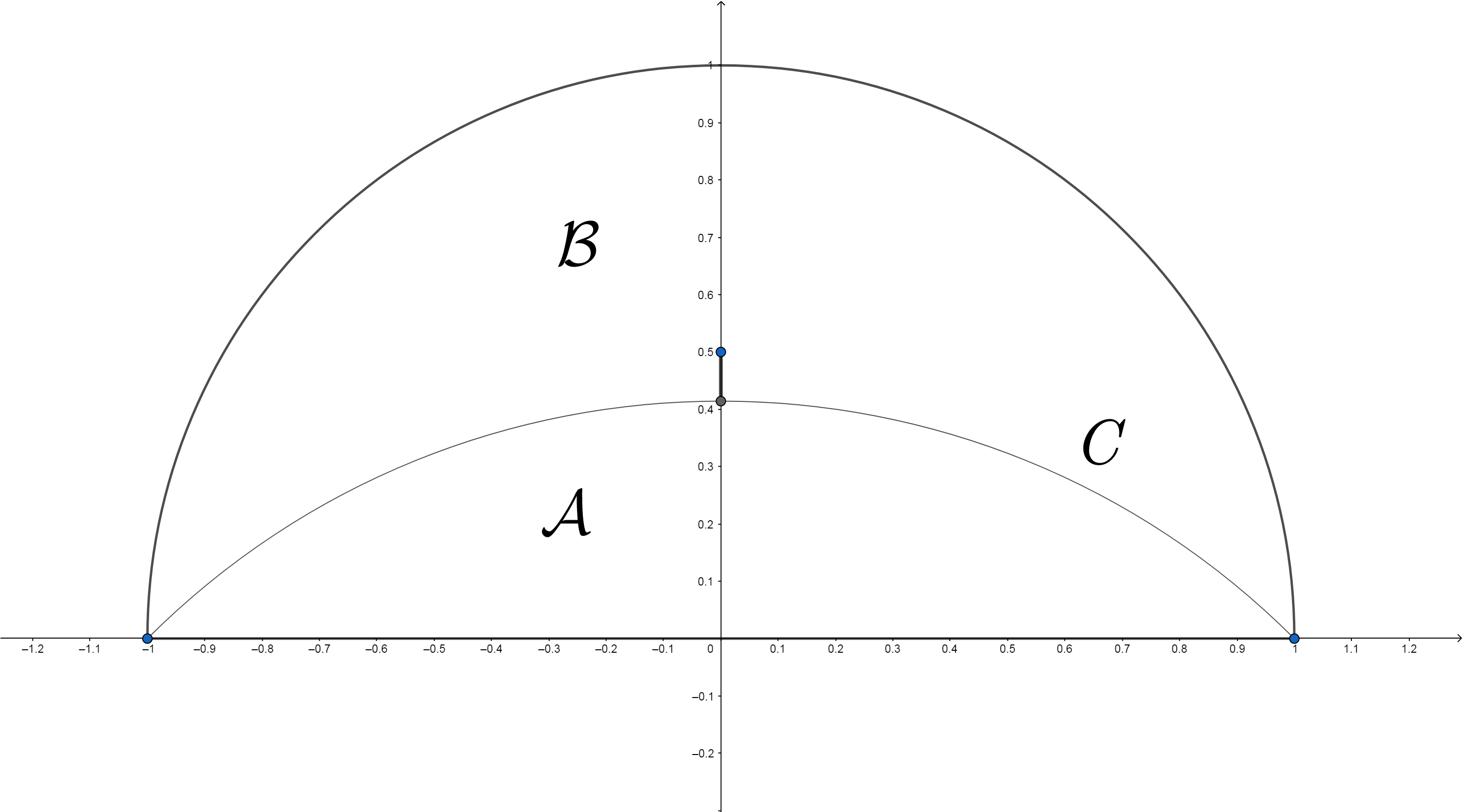}\caption{Reflection over a circle}
\end{figure}

\begin{example} \label{isotri}
Now let $U$ be an isosceles triangle with vertices at $-1, 1$, and $Ni$ with $N>0$. It will be convenient for us to index $U$ by the angles at $1$ and $-1$, so if we let $\thh$ be this angle then $N = \tan \thh$. Proposition \ref{prop:full.sym.axis} tells us that all $p$-th centers lie on the imaginary axis. We have seen already from the example discussed in connection with Proposition \ref{prop:part.sym.axis.conformal} that all $p$-th centers must lie below $\frac{M}{2}i$, however we will show now how this can be improved. Let $B$ be the angle bisector of one of the base angles of $U$. $B$ is a partial symmetry axis of $U$, with symmetric side given by the component of $U \backslash B$ corresponding to the shorter side of the triangle. Thus, if $\thh > \frac{\pi}{3}$, then all $p$-th centers must lie above $B$, while if $\thh < \frac{\pi}{3}$, then all $p$-th centers must lie below $B$. Now let $M$ be the perpendicular bisector of the edge connecting $1$ to $Mi$ (this is often referred to as the {\it mediator}). This is also a partial symmetry axis of $U$, and the symmetric side is the component of $U \backslash M$ which does not contain $-1$. Thus, if $\thh > \frac{\pi}{3}$, then all $p$-th centers must lie below $M$, while if $\thh < \frac{\pi}{3}$, then all $p$-th centers must lie above $M$. Thus, the intersections of $M$ and $B$ with the imaginary axis provide upper and lower bounds for all $p$-th centers, although which is the upper bound and which is the lower bound depends on $\thh$. The following figure demonstrates this phenomenon ($\Delta$ denotes the imaginary axis).

\begin{figure}[H]
\centering{}\includegraphics[width=14cm,height=12cm,keepaspectratio]{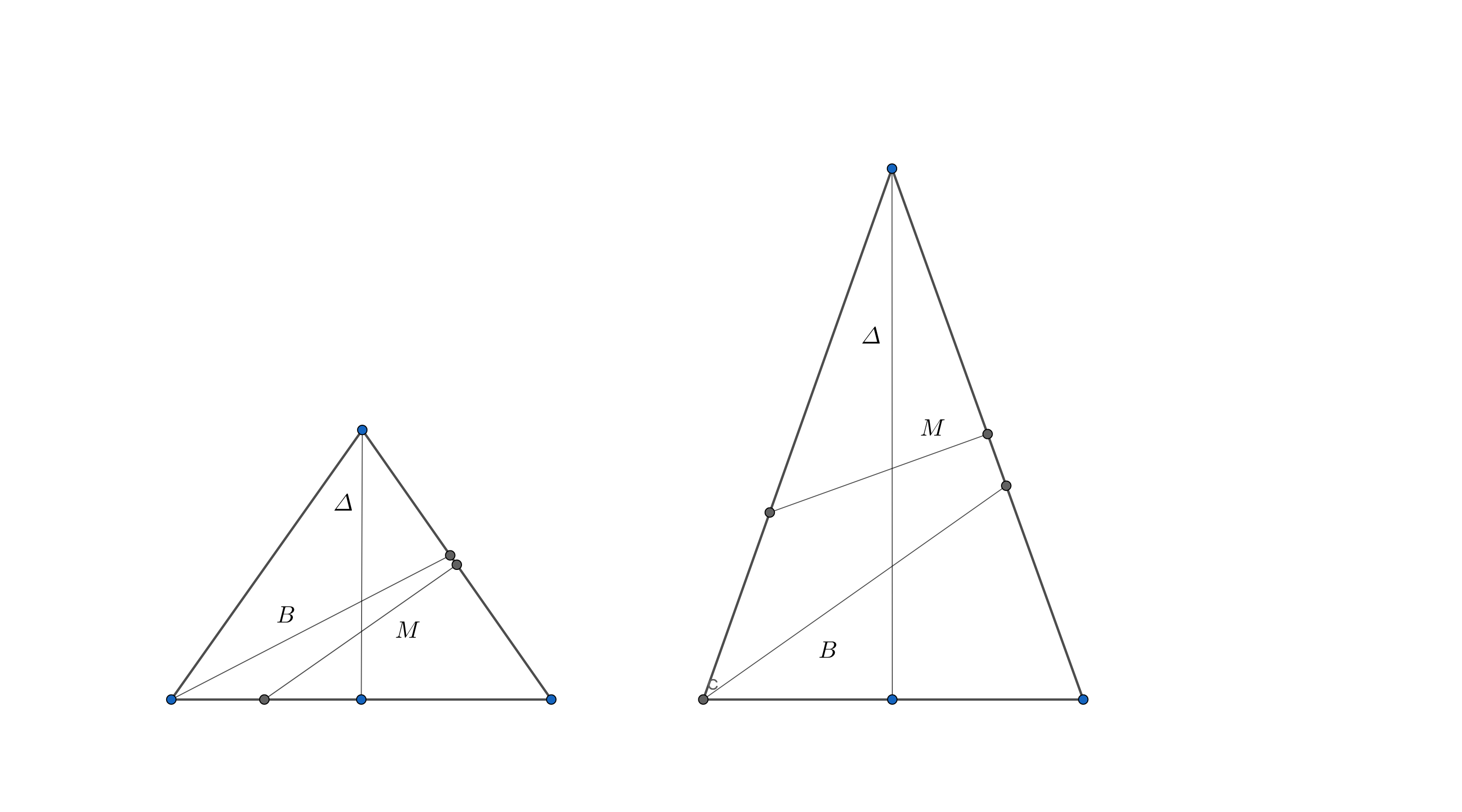}\caption{The angle bisector and mediator}
\end{figure}

Naturally they coincide at $\thh = \frac{\pi}{3}$. It can be checked that, regardless of $\thh$, this gives a better upper bound than $\frac{N}{2}$, which was given by reflection over $\{Im(z) = \frac{N}{2}\}$. A bit of Euclidean geometry shows that the intersection of $B$ with the imaginary axis is at the point $\tan(\frac{\thh}{2})i$, and the intersection of $M$ with the imaginary axis is at the point $(\frac{1}{\tan \thh}- \frac{1}{\sin 2\thh})i$. Furthermore, we always have as a lower bound the intersection of the imaginary axis and the circle passing through $1$ and $-1$, making an angle of $\frac{\thh}{2}$ with the real axis; this follows from Corollary \ref{circreflect} as above. This point is $\frac{1-\cos(\frac{\thh}{2})}{\sin(\frac{\thh}{2})}$. The following graph shows these upper and lower bounds; all $p$-th centers must lie in the regions labelled $\Omega$.

\begin{figure}[H]
\centering{}\includegraphics[width=9cm,height=9cm,keepaspectratio]{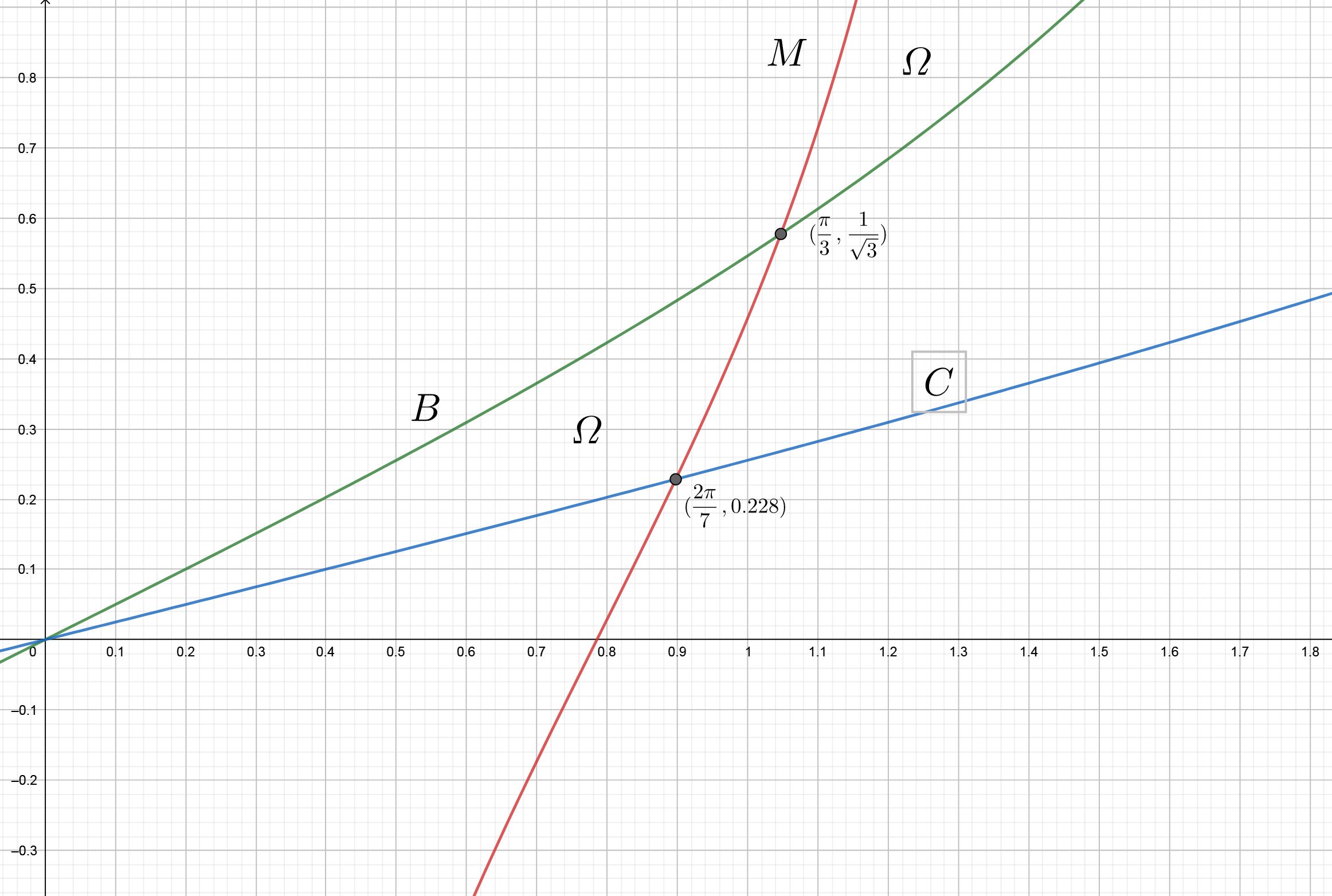}\caption{Upper and lower bounds for $p$-th centers}
\end{figure}

{\bf Remark:} We believe that a better lower bound can be achieved through numerical conformal mapping; more on this in the final section.

\end{example}

\begin{example}
\label{annulus}Let $\mathcal{A}_{r,R}$ be the annulus $\{r<|z|<R\}$.
 Then all $p$-centers lie in $\{\sqrt{{\scriptstyle rR}}<|z|<\frac{R+r}{2}\}$
\end{example}

\begin{proof}
Consider $f(z)=\frac{rR}{\overline{z}}$ which maps $\mathcal{A}_{r,R}$
to itself. Under the same notation as in Proposition \ref{map U to U} we have
\[
\begin{alignedat}{1}\Omega & :=\{\sqrt{rR}<|z|<R\}\\
\Lambda & :=\{|z|=\sqrt{rR}\}
\end{alignedat}
\]
and we can check easily that $f$ satisfies the requirements of Proposition \ref{map U to U}).
Therefore we can eliminate $U\setminus\Omega$ from consideration, and we obtain the lower bound $\sqrt{{\scriptstyle rR}}$.
In order to get the upper bound $\frac{R+r}{2}$ we can see, as illustrated
by \textbf{Fig}. (\ref{fig:-is-limited}), that the line $\Delta$
is a partial symmetry axis. The result follows.

\end{proof}

\begin{figure}[H]
\centering{}\includegraphics[width=9cm,height=9cm,keepaspectratio]{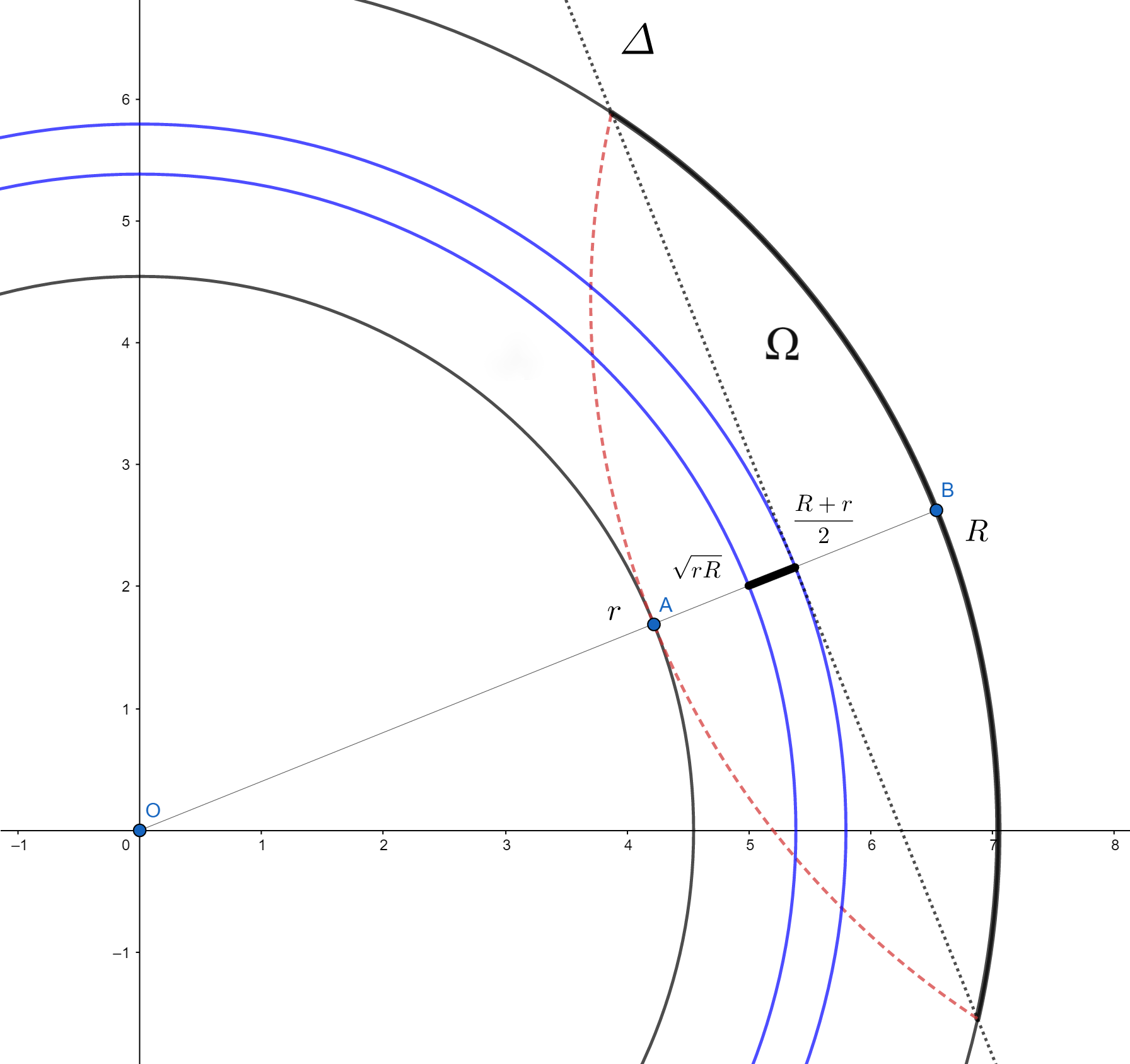}\caption{ Bounds for the annulus \label{fig:-is-limited}}
\end{figure}

{\bf Remarks:} \begin{itemize} \label{}

\item Another way to get the same lower bound as above is to note that Proposition \ref{prop:part.sym.axis.conformal} extends to non-injective maps in suitable situations. We may use the map
\[
f:\begin{alignedat}{1}\left\{ \ln r<Re(z)<\ln R\right\}  & \longrightarrow\mathcal{A}_{r,R}\\
z & \longmapsto e^{z}
\end{alignedat}
\]

and apply this extension of Proposition \ref{prop:part.sym.axis.conformal} with reflection axis $\{Re(z) = \frac{\ln R + \ln r}{2}\}$ in order to obtain the result.

\item It should be mentioned that an explicit formula for the first moment can be obtained by Dynkin's formula, and it is

\[
E_{z}(\tau_{\mathcal{A}_{r,R}})=\frac{R^{2}\ln\left(\frac{|z|}{r}\right)-r^{2}\ln\left(\frac{|z|}{R}\right)}{\ln\frac{R}{r}}-|z|^{2}.
\]

This can be shown to be maximal at $|z|=\sqrt{\frac{R^{2}-r^{2}}{2\ln\frac{R}{r}}}$. Our estimates are therefore not necessary for the first moments, but as an aside we obtain the inequality

\[
\sqrt{Rr}<\sqrt{\frac{R^{2}-r^{2}}{2\ln\frac{R}{r}}}<\frac{R+r}{2}.
\]

Setting $a=R^2$, $b=r^2$, and squaring the inequalities gives the following:

\[
\sqrt{ab}<\frac{a-b}{\ln\frac{a}{b}}<\frac{a+b + 2\sqrt{ab}}{4} \leq \frac{a+b}{2}.
\]

The quantity $\frac{a-b}{\ln\frac{a}{b}}$ is of great importance in the study of heat flow, and in that context it is known as the {\it logarithmic mean temperature difference}, or LMTD (see \cite{kern}). We have therefore given a new proof of the fundamental fact that the LMTD lies between the arithmetic and geometric means, and in fact have proved that the upper bound can be lowered to the arithmetic mean of the arithmetic and geometric means.

\end{itemize}

\begin{example} \label{hyperbola}
Let $\mathscr{H}$ be the region $\{|x|>|y|, x^2-y^2<1\}$; this is the region bounded by the lines $y = \pm x$ and the hyperbola $x^2-y^2=1$, see the figure below.

It is perhaps not obvious for which $p$ we have $E_{w}[\tau_{\mathscr{H}}^p]< \ff$, however we can show that $E_{w}[\tau_{\mathscr{H}}^p]< \ff$ for any $p>0$ and $w \in \mathscr{H}$, as follows. $\mathscr{H}$ is contained in the union of two infinite strips which are othogonal. Any strip has all moments of its exit time finite: Brownian motion is rotation invariant, so the moments are the same as for a horizontal strip, and these moments in turn are the same as for a one-dimensional Brownian motion from a bounded interval, since that is what we obtain when we project the Brownian motion onto the imaginary axis; these moments are well known to be finite for all $p$, and in fact they can be calculated explicitly for integer $p$ using the Hermite polynomials. We would like to conclude that the union of these two strips must then have finite $p$-th moment, however easy examples show that it is not necessarily the case that the union of two domains with finite $p$-th moment must itself have finite $p$-th moment. A method does exist for reaching the desired conclusion, though, and it is contained in Theorem 3 and Lemmas 1 and 2 of \cite{mebabyme}. It is strightforward to verify that our infinite strips satisfy the required conditions: their interesection is bounded, and boundary arcs intersect at non-zero angles. Therefore the exit time for their union has finite $p$-th moments for all $p$, and thus so does $\mathscr{H}$. See \cite{mebabyme} for details.

Now let us see how our methods can be used to localize the $p$-th centers. The real axis is an axis of symmetry, however the domain is not $\Delta$-convex, so we may not apply Proposition \ref{prop:full.sym.axis}. However, all $p$-centers do lie on the real axis, and we may prove this as follows. The map $f(z) = \sqrt{z}$ maps the strip $\{0<Re(z)<1\}$ conformally onto $\mathscr{H}$. Any horizontal line can be used in Proposition \ref{prop:part.sym.axis.conformal}, and we note that $|f'(z)| = \frac{1}{2\sqrt{|z|}}$ is monotone decreasing in $|z|$ and therefore in $|Im(z)|$. Thus, all $p$-th centers must lie on the image of the real axis under $f$, which is again the real axis. So we need only consider point on $\RR$. The line $\{Re(z)=\frac{1}{2}\}$ is a partial symmetry axis, which gives a lower
bound of $\frac{1}{2}$ for all $p$-th centers. For an upper bound, note that $\{Re(z) = \frac{1}{2}\}$ is another axis of symmetry of $\{0<Re(z)<1\}$, and the monotonicity of the derivative shows again via Proposition \ref{prop:part.sym.axis.conformal} that we only need to look in the region $f(\{0<Re(z)<\frac{1}{2}\})$, which is the region $\{x^2-y^2< \frac{1}{2}\}$ inside $U$. This gives an upper bound of $\frac{1}{\sqrt{2}}$ on the real axis. Thus, all $p$-centers lie on $\{Im(z) = 0, \frac{1}{2} < Re(z) < \frac{1}{\sqrt{2}}\}$. This set is in bold below.

\begin{figure}[H]
\centering{}\includegraphics[width=9cm,height=9cm,keepaspectratio]{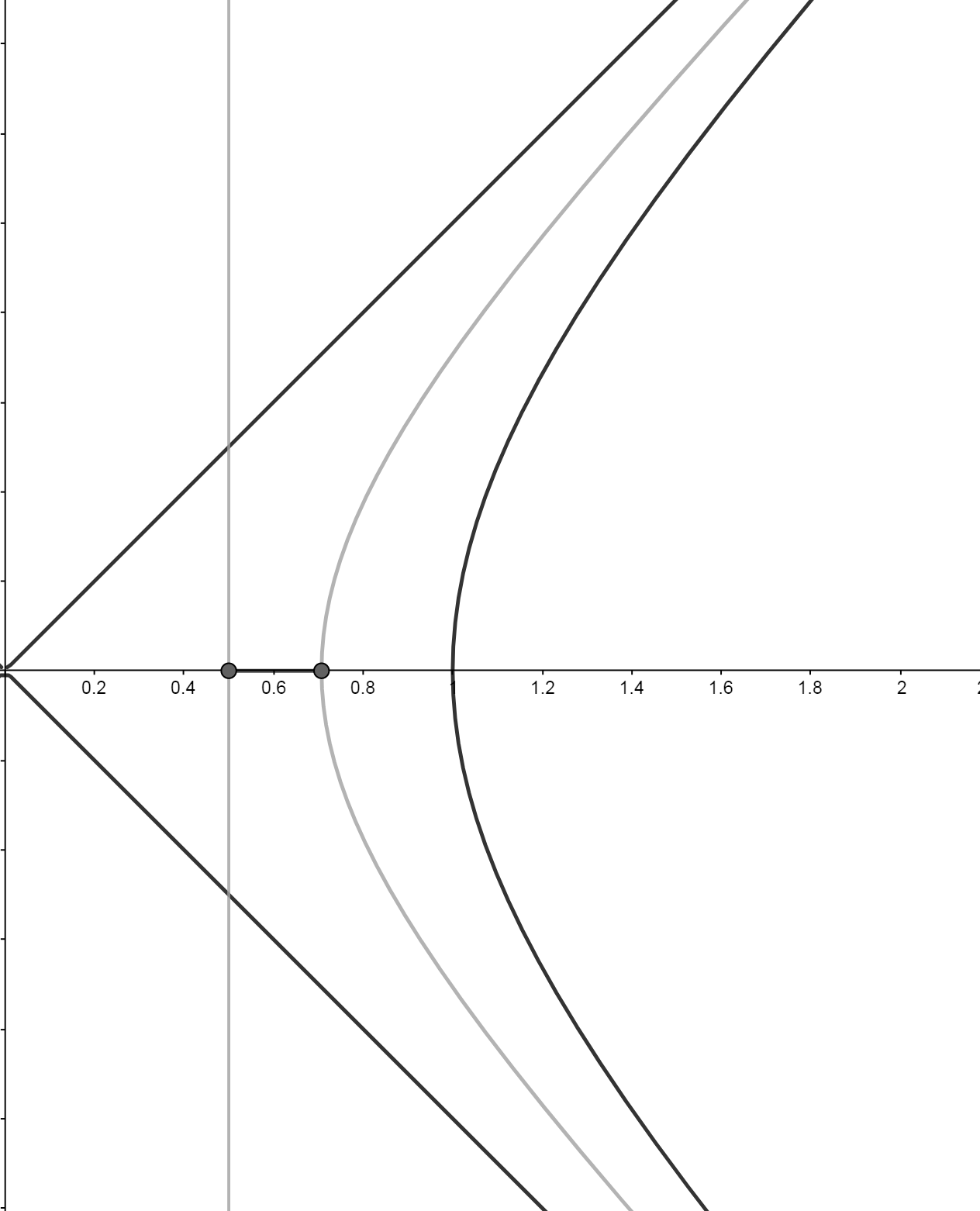}\caption{ Bounds for the hyperbolic region } \label{hyperbolapic}
\end{figure}

\end{example}

\begin{example} \label{crescent}
Let $\mathscr{C}_{R}$ be the crescent-like shape limited by the
two circles $\{|z-\frac{1}{2}| = \frac{1}{2}\}$ and $\{|z-\frac{R}{2}| = \frac{R}{2}\}$ (see the figure below). $\mathscr{C}_{R}$ is the image of the region $\{\frac{1}{R}< Re(z) < 1\}$ under the conformal map $f(z) = \frac{1}{z}$. Note that $|f'(z)| = \frac{1}{|z|^2}$ is monotone decreasing in $|z|$, so by the same argument as in Example \ref{hyperbola} all $p$-th centers lie on the real axis. Furthermore $\{Re(z) = \frac{R+1}{2}\}$ is a partial symmetry axis for $\mathscr{C}_{R}$, and this allows us to eliminate the region $\{Re(z) > \frac{R+1}{2}$ from consideration. We may also use axis of symmetry $\{Re(z) = \frac{1+\frac{1}{R}}{2}\}$ in order to conclude via Proposition \ref{prop:part.sym.axis.conformal} that we can exclude the region $f(\{\frac{1}{R}<Re(z)<\frac{1+ \frac{1}{R}}{2}\})$, which is the region $\mathscr{C}_{R} \cap \{|z-\frac{R}{R+1}| < \frac{R}{R+1}\}$, in the search for $p$-th centers. We see that all $p$-centers lie on the interval $\{Im(z)=0, \frac{2R}{R+1} < Re(z) < \frac{R+1}{2}\}$, which is in bold below for $R=2$.

\begin{figure}[H]
\centering{}\includegraphics[width=11cm,height=11cm,keepaspectratio]{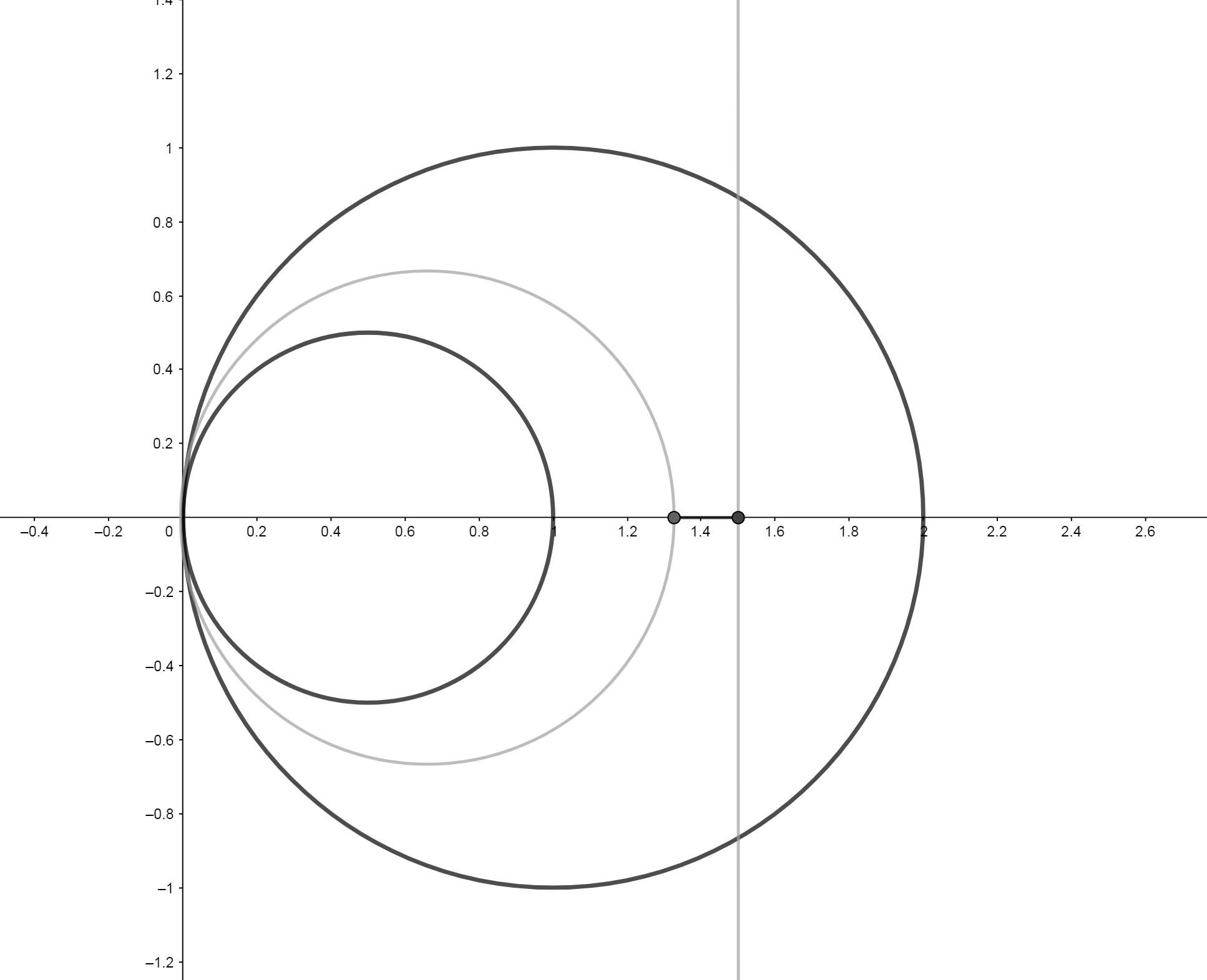}\caption{ Bounds for the crescent region \label{fig:-is-limited2}} \label{crescentpic}
\end{figure}

\end{example}

\section{Concluding remarks}

As remarked earlier, in Figure \ref{trireflect} the regions $\AA$ and $\BB$ do not fill all of $U$, and it is natural to search for a a better bound by finding a conformal map that fills the entire domain. Let us consider the Schwarz-Christoffel transformation sending the unit disk to $U(\thh)$ given by (see \cite[Ch. 2]{tref})

$$
f(z) = A + C \int_{0}^{z} (1-w)^{\frac{\thh}{\pi}-1}(1+w)^{\frac{\thh}{\pi}-1} (1+iw)^{\frac{-2\thh}{\pi}}dw
$$

\noindent for appropriate choices of constants $A$ and $C$; note that this is chosen so that $1$ and $-1$ are mapped to the base angles, and $i$ is mapped to the top angle. We have

$$
|f'(z)| = |C||1-z|^{\frac{\thh}{\pi}-1}|1+z|^{\frac{\thh}{\pi}-1} |1+iz|^{\frac{-2\thh}{\pi}}.
$$

It can be checked then that $|f'(z)| > |f'(\bar z)|$ whenever $Re(z)>0$, so Proposition \ref{prop:part.sym.axis.conformal} implies that no $p$-th centers can be found in the image of $\DD \cap \{Re(z)<0\}$. From this point, a numerical method can be employed, and the resulting bound should improve the one we found, if desired.

As was mentioned in connection with $\Delta$-convexity, there are some purely geometrical consequences of our results. In that context, the following may be proved.

\begin{prop} \label{}
Suppose a domain $U$ is $\Delta$-convex with respect to two parallel symmetry axes. Then we can find $a \in [-\ff,\ff)$ and $b \in (-\ff,\ff]$ so that $U$ is a rotation of the domain $\{a < Re(z) < b\}$; in other words, $U$ is all of $\CC$, is a half-plane, or is an infinite strip.
\end{prop}

As a corollary of this, and of our probabilistic results above, we obtain the following.

\begin{corollary} \label{}
Suppose $U$ is a domain which is not all of $\CC$, a half-plane, or an infinite strip. Then if there are multiple axes of symmetry to which $U$ is $\Delta$-convex then they all meet at a unique point.
\end{corollary}

\section{Acknowledgements}

The authors would like to thank Paul Jung, Fuchang Gao, and Lance Smith for helpful conversations. We are also grateful to several anonymous referees for helpful comments.

\bibliographystyle{IEEEtranN}
\bibliography{bibfile,\string"C:/Users/mbou0007/Google Drive/Maheref\string"}

\end{document}